\documentclass[a4paper,11pt]{amsart}

\usepackage{amsmath,amssymb,amsfonts,amsthm}
\usepackage{eqnarray}
\usepackage{tikz, tikz-cd}
\usepackage{hyperref}
\usepackage{caption, subcaption} 
\usepackage{enumerate}
\usepackage{verbatim}

\usepackage{thmtools}
\usepackage{thm-restate}

\usepackage{todonotes}

\newcommand{\vs}{\vspace{0.4cm}}
\newcommand{\OCP}{(\mathbb{O} \otimes \mathbb{C})\mathbb{P}^{2}}
\newcommand{\OHP}{(\mathbb{O} \otimes \mathbb{H})\mathbb{P}^{2}}
\newcommand{\OOP}{(\mathbb{O} \otimes \mathbb{O})\mathbb{P}^{2}}
\newcommand{\Z}{\mathbb{Z}}

\newcommand{\X}{\mathfrak{X}}

\newtheorem{thm}{Theorem}[section]
\newtheorem{lem}[thm]{Lemma}
\newtheorem{prop}[thm]{Proposition}

\newtheorem{introthm}{Theorem}

\theoremstyle{definition}
\newtheorem{df}[thm]{Definition}

\newtheorem{ttable}[thm]{Table}

\theoremstyle{remark}

\newtheoremstyle{break}
{\topsep}{\topsep}%
{}{}
{\bfseries}{}%
{\newline}{}%
\theoremstyle{break}

\numberwithin{equation}{section}

\begin{document}

\title[Orientability of high-dim. manifolds with odd Euler char.]{Orientability of high-dimensional manifolds with odd Euler characteristic}

\author[R. S. Hoekzema]{Renee S. Hoekzema}
\address{Department of Mathematics, University of Oxford}

\email{hoekzema@maths.ox.ac.uk}



\date{}

\dedicatory{}

\maketitle

\begin{abstract}
	We call a manifold $k$-orientable if the $i^{th}$ Stiefel-Whitney class vanishes for all $i< 2^k$ ($k\geq 0$), generalising the notions of orientable (1-orientable) and spin (2-orientable). In \cite{Hoekzema2017} it was shown that $k$-orientable manifolds have even Euler characteristic (and in fact vanishing top Wu class), unless their dimension is $2^{k+1}m$ for some $m\geq 1$. This theorem is strict for $k=0,1,2,3$, but whether there exist 4-orientable manifolds with an odd Euler characteristic is an open question. This paper discusses the question of finding candidates for such a manifold $\X^{32m}$.
	
	As part of our investigation we study the example of the three exceptional symmetric spaces known as Rosenfeld planes, which have odd Euler characteristic and are of dimension 32, 64 and 128. We perform computations of the action of the Steenrod algebra on the mod 2 cohomology of the first two of these manifolds with the use computer calculations. The first Rosenfeld plane, $\OCP$, is 2-orientable but not 3-orientable and thus not an example of $\X^{32}$. We show that the second Rosenfeld plane $\OHP$ is 3-orientable and we present a condition under which is may be 4-orientable if the action of the Steenrod algebra is established further, and therefore it remains a potential candidate for $\X^{64}$.
	No other clear candidate manifolds for $\X^{32m}$ or in particular candidates for $\X^{32}$ are known to the author.
	
\end{abstract}

\tableofcontents

\section{Introduction}

It is a straightforward consequence of Poincar\'e duality in mod 2 homology that all manifolds of odd dimension have a vanishing Euler characteristic.
For orientable manifolds, the Euler characteristic is moreover even in dimensions that are not a multiple of 4.
Ochanine's theorem for the divisibility of the signature of spin manifolds implies that the Euler characteristic of spin manifolds is even unless the dimension is a multiple of 8 (\cite{Ochanine1981}). This sequence of statements was generalised in \cite{Hoekzema2017}.

\begin{df}[\cite{Hoekzema2017}]
	We call a manifold $k$-orientable if $w_i=0$ for $0<i<2^k$.
\end{df}

\begin{thm}[\cite{Hoekzema2017}]\label{mainthm}
	A $k$-orientable manifold or Poincar\'e complex $M$ ($k \geq 0$) has an even Euler characteristic $\chi (M)$ if its dimension is not a multiple of $2^{k+1}$. 
\end{thm}

This paper discusses the question whether this theorem is strict: does a $k$-orientable manifold with odd Euler characteristic ($\chi$) exist in all dimensions admitted by Theorem \ref{mainthm}, for all values of $k$? For $k=0,1,2,3$, the answer is yes, and the exemplifying manifolds are summarised in the following table.

\vs

\begin{tabular}{|l|l|c|c|}
	\hline
	$k$ & corresponds to & \bf{dimensions with} 	& \bf{$k$-orientable}\\
	& & \bf{odd $\chi$ possible} & \bfseries{manifolds with} \\
	& & & \bf{odd $\chi$}\\
	\hline
	0 & any manifold & $2m$ & $\mathbb{RP}^{2m}$\\
	1 & orientable manifolds & $4m$ & $\mathbb{CP}^{2m}$\\
	2 & spinable manifolds & $8m$ & $\mathbb{HP}^{2m}$\\
	3 & implied by stringable & $16m$ & $(\mathbb{OP}^2)^m$\\
	4 & implied by fivebraneable & $32m$ & $\X^{32m}$? \\	\hline
\end{tabular}

\vs

Whether a 4-orientable manifold with odd Euler characteristic exists, here referred to as $\X^{32m}$, is currently an open question.
By theorem \ref{mainthm}, such a manifolds would need be of a dimension divisible by 32. 
One might look for examples of highly orientable manifolds with odd Euler characteristic by considering special classes of manifolds that generalise the sequence of projective planes over division algebras.

\subsection*{Rational projective planes}
In \cite{Kennard2017} the existence of rational projective planes in different dimensions is considered. Rational projective planes are simply connected, smooth, closed manifolds with rational cohomology isomorphic to $\mathbb{Q}[a]/\langle a^3 \rangle$, and thus in particular have odd Euler characteristic. Kennard and Su show that these manifolds exist in dimensions 4, 8, 16, 32, 128 and 256, but there are none in any dimension between 256 and $2^{13}$, with the possible exception of 544, 1024, 2048, 4160 and 4352. Whether there exist infinitely many rational projective planes remains an open question. However, the authors did show that no spinable rational projective planes exist in dimensions other than 8 and 16 (\cite{Kennard2017}, Theorem C). Hence all rational projective planes besides those in the same dimensions as $\mathbb{HP}^2$ and $\mathbb{OP}^2$ are 1-orientable but not 2-orientable, and hence this class of manifolds does not give rise to any candidates for $\X^{32m}$.

\subsection*{Rosenfeld planes}
A different direction in which to generalise the notion of projective planes is given by considering symmetric spaces. All projective spaces over different division algebras are in particular symmetric spaces of different Lie groups. Irreducible compact symmetric spaces were classified by Cartan and there are seven infinite families and twelve exceptional spaces in this classification. Recently, a uniform construction of all compact symmetric spaces was found, describing them as different types of Grassmannians using the four division algebras and the Freudenthal magic square (\cite{Huang2011}). Three particularly interesting symmetric spaces in this context are given by the Rosenfeld planes of dimensions 32, 64 and 128. We here denote them as $\OCP$, $\OHP$ and $\OOP$, but in Cartan's classification scheme they are denoted EIII, EVI and EVIII respectively.
They are also known as the bioctonionic projective plane, the quateroctonionic projective plane and the octooctonionic projective plane.
These manifolds have an odd Euler characteristic. They can be described as Lie group quotients of the exceptional groups $E_6$, $E_7$ and $E_8$ in the following way (\cite{Baez2002}): 
\begin{eqnarray}
\OCP &=& E_6/\left((Spin(10)\times U(1))/\mathbb{Z}/4\right)\nonumber\\
\OHP &=& E_7/\left( (Spin(12)\times Sp(1))/\mathbb{Z}/2\right)\nonumber\\
\OOP &=& E_8/\left( Spin(16)/\mathbb{Z}/2\right)\nonumber
\end{eqnarray}

The cohomology of EIII or $(\mathbb{O} \otimes \mathbb{C})\mathbb{P}^{2}$ is known integrally, for EVI or $(\mathbb{O} \otimes \mathbb{H})\mathbb{P}^{2}$ the cohomology is known modulo 2 and for EVIII or $(\mathbb{O} \otimes \mathbb{O})\mathbb{P}^{2}$, as well as EV and EIX in Cartan's classification scheme, the cohomology remains undetermined.


\subsection*{Results}

In this paper we consider the orientability of $\OCP$ and $\OHP$ by considering the action of the Steenrod algebra on their mod 2 cohomology. 
We compute the action of specific Steenrod squares with target the top cohomology for both these manifolds. The computations make use of a program written in Wolfram Mathematica by the author that aids in the expansion of the Cartan formula on large products.
In section \ref{OCP} we show

\begin{introthm}
	$(\mathbb{O} \otimes \mathbb{C})\mathbb{P}^{2}$ is 2-orientable and not 3-orientable.
\end{introthm}

We therefore establish that $\OCP$ is not a candidate for $\X^{32}$. Moving on to the next Rosenfeld plane, we show in section \ref{OHP}

\begin{introthm}
	$\OHP$ is 3-orientable. 
\end{introthm}

However, the question whether $\OHP$ presents a candidate for $\X^{64}$ cannot be resolved within our present-day understanding of the cohomology of the manifold.

\begin{introthm}
	$\OHP$ is 4-orientable if $\beta''+ \nu''=1$, where $\beta''$ is the prefactor of $y_2 y_3^2 y_{12}$ in $Sq^8 y_{12}$ and $\nu''$ is the prefactor of $y_2 y_3^2 y_{20}$ in $Sq^8 y_{20}$. 
\end{introthm}

The prefactors $\beta''$ and $\nu''$ are part of the indeterminacy in our current knowledge of the action of the Steenrod algebra on the mod 2 cohomology of $\OHP$ as established in \cite{Nakagawa2001}.

The following table now summarises our updated knowledge of the orientability of the Rosenfeld planes and their candidacy for finding an $\X^{32m}$.

\vs

\begin{tabular}{|l|l|l|l|}
	\hline
	Rosenfeld  & Cartan's & $k$-orientable  & Status 	\\
	plane & classification & for $k$: &  \\
		\hline
	$\OCP $& EIII & 2 not 3 & Not a candidate for $\X^{32}$\\
	$\OHP $& EVI & 3 maybe 4 & Potential candidate for $\X^{64}$\\
	$\OOP $& EVIII & unknown  & Potential candidate for $\X^{128}$, \\
	& & & currently cohomology unknown\\
\hline
\end{tabular}
\vs

No other clear sources of candidate manifolds for $\X^{32m}$ are known to the author. In particular, one can wonder whether if 
an $\X^{32}$ exists, a 4-orientable a manifold with odd Euler characteristic in the lowest possible dimension.


\subsection*{Outline}
In section \ref{section:methods} we give an overview of the computational methodology for establishing $k$-orientability of a high-dimensional manifold when one is presented with its mod 2 cohomology and the (partial) action of the Steenrod algebra on this cohomology. Section \ref{OCP} discusses application of this methodology to $\OCP$. As intermediate computational results we list an additive basis for the integral cohomology of $\OCP$ and the corollary that its signature is 3. Section \ref{OHP} discusses computations and results for $\OHP$. The 64 additive groups and generators in terms of the ring generators are listed in subsection \ref{additiveOHP}, and a full list of relations can be found online\footnote{Supplementary materials can be found online at \href{http://reneehoekzema.nl/Mathematics.html}{reneehoekzema.nl/Mathematics}.}. 
	Moreover we give the explicit calculations of the action of $Sq^3$, $Sq^5$, $Sq^6$ and $Sq^7$ on the ring generators. We list a number of explicit computational results in this paper in the hope that they will make the methodology for determining $k$-orientability as well as the study of the first two Rosenfeld planes accessible to interested researchers and students with little computational background.

\subsection*{Acknowledgements} It is a pleasure to acknowledge contributions to this projects arising from discussions and email conversations with Andrew Ranicki, Masaki Nakagawa, Robert Bruner, Andr\'e Henriques, Christopher Douglas and Ulrike Tillmann. The author is particularly indebted to Robert Bruner for helping with the implementation of Magma Computational Algebra System to compute an additive basis for the cohomology of $\OHP$, which was indispensable for the calculations presented in section \ref{OHP}.

This research was supported by the EPSRC, the Hendrik Mullerfonds, the Fundatie van de Vrijvrouwe van Renswoude, the Oxford Mathematical Institute and Trinity College Oxford. Furthermore, the author would like to thank the Max Planck Institute for Mathematics in Bonn for its hospitality and financial support during the editing process.

\section{Determining $k$-orientability from the action of the Steenrod algebra on mod 2 cohomology}\label{section:methods}

\subsection{Link to Wu classes}

In \cite{Hoekzema2017} it was shown that $k$-orientability can be established by considering the values of the Wu classes rather than the Stiefel-Whitney classes. Given a manifold $M$ of dimension $n$ and $i\leq n/2$, the Wu class $v_i \in H^i (M; \Z/2) $ is the class such that squaring to the top dimension corresponds to cupping it with the corresponding Wu class:
\begin{equation}
v_i \smile x_{n-i} = Sq^i (x_{n-i}) \in H^{n} (M; \Z/2),
\end{equation}

for any class $x_{n-i} \in H^{n-i} (M; \Z/2)$.

It was shown in \cite{Hoekzema2017} that rather than requiring the lowest Stiefel-Whitney classes to vanish, one can equivalently define $k$-orientability by requiring that Wu classes in the same degrees vanish.

\begin{lem}[\cite{Hoekzema2017}]\label{orientableWu}
A manifold is $k$-orientable if and only if $v_i=0$ for $0<i \leq 2^{k-1}$.
\end{lem}

Hence, whether a manifold is $k$-orientable can be established by considering the action of the Steenrod algebra on its cohomology. In particular, in order to determine whether an $n$-dimensional manifold is $k$-orientable, we need to evaluate the Steenrod squares of degree $2^j$ on all generators in degree $n-2^j$ for $j = 0, 1, ..., k-1$. Steenrod squares are generated by those of degree a power of two, so if all these operations vanish then we can conclude that the manifold is $k$-orientable.

\subsection{Establishing an additive basis}
The integral cohomology of $\OCP$ was determined in \cite{Toda1974}. The mod 2 cohomology of $\OHP$ was determined in \cite{Nakagawa2001}. Cohomology is generally presented as a graded ring, with ring generators and ring relations. However, in order to do computations we need to expand this compact information and determine what the groups look like individually, establishing an additive basis and relations in terms of homogeneous monomials in the ring generators in every degree.
In the case of $\OCP$ we do this in Wolfram Mathematica by generating a list of homogeneous monomials in every degree, computing homogeneous parts of the ideal of relations and solving this system of equations degreewise. The cohomology ring of $\OCP$ has two generators of degree 2 and 8 and two relations in degrees 18 and 24. The notebook can be found online$^\text{1}$, and the resulting additive basis plus all relations between monomials is shown in Table \ref{additiveOCP}.
The action of the Steenrod algebra on the cohomology of $\OCP$ was determined in \cite{Ishitoya1992}, and using this information, the rest of the computations for $\OCP$ in this paper could be done by hand.

The mod 2 cohomology of $\OHP$ has five generators in degree 2, 3, 12, 16 and 20, and twelve relations in degrees 9, 18, 19, 26, 27, 28, 36, 40, two in 44, 48 and 52. This means that for example in the top dimension 64, there are 123 homogeneous monomials in the generators and 245 relations amongst them. 
In this case Mathematica could no longer handle computations. However, Magma computational algebra system is a program specialised in algebraic computations such as these and was able to carry out the computation with ease.  Apart from establishing an additive basis, which is shown in Table \ref{additiveOHP}, a list of 49 pages of relations between monomials is listed online$^\text{1}$. For the case of dimension 64, this information is presented in section \ref{64monomials}. There are 123 monomials in the ring generators in this degree, of which eleven equal the fundamental class and the other 112 equal zero. 
This explicit information is required for the calculations in section \ref{OHP}.

\subsection{Expanding the Cartan formula}

In order to establish $k$-orientability we want to evaluate Steenrod operations on high-dimensional classes, which are written as monomials in ring generators. To do so we use the Cartan formula:
\begin{equation}
Sq ^n (x \smile y ) = \sum_{i+j =n} (Sq^i x) \smile (Sq^j y).
\end{equation}

On a larger product we can rewrite the formula as follows:
\begin{equation}
Sq^n \left(\prod_{j=1}^k x_{j} \right)  = \sum_{\{p\}} \prod_{j=1}^k Sq^{p_j} x_j,
\end{equation}

where $\{p\}$ ranges over the collection of ordered partitions of $n$ into $k$ numbers between 0 and $n$. A program was written in Mathematica by the author to aid in the expansion of this formula. The program can be found online$^\text{1}$ and can be applied more generally.
The use of the Mathematica code is essential as the expansion of the Cartan formula on larger products with high squares involves a large number of terms. For example, if we want to apply $Sq^8$ on a product of twelve generators, the sum over partitions contains 75,582 terms, each of which is a product of twelve squaring operations on a class. In the case of $\OHP$, each one of these operations has up to five terms when fully expanded in terms of monomials in the ring generators. In the end most terms appear an even number of times and hence vanish.

\subsection{Computing lower Steenrod operations}
Nakagawa partially established the action of $Sq^1$, $Sq^2$, $Sq^4$ and $Sq^8$ on all generators of the mod 2 cohomology of $\OHP$ in the process of determining this cohomology in \cite{Nakagawa2001}. 
When we compute
\begin{equation}
Sq^n \left(\prod_{j=1}^k x_{j} \right)
\end{equation}
by applying the Cartan formula, the output is a linear combination of products of terms of the form $Sq^i x_j$, for $x_j$ a ring generator of the cohomology, where $i$ ranges between 0 and $n$. Hence we need to determine values of the squaring operations that are not a power of two. We do so by using the Adem relations: for all $a$ and $b$ such that $a < 2b$, we have
\begin{equation}
Sq^a  Sq^b =  \sum_{c=0}^{\lfloor a/2\rfloor} {b-c-1 \choose a-2c} Sq^{a+b-c}  Sq^c.
\end{equation}

In particular we obtain

\begin{eqnarray}
Sq^{2n+1} &=& Sq^1 Sq^{2n} \\
Sq^6 &=& Sq^5 Sq^1 + Sq^2 Sq^4. 
\end{eqnarray}

We use these equations to compute the action of $Sq^3$, $Sq^5$, $Sq^6$ and $Sq^7$ on all generators of the cohomology of $\OHP$ in section \ref{smallsquares}.

\subsection{Determining $k$-orientability}

Concluding the above, a manifold $M$ is $k$-orientable if 
\begin{equation}
Sq^{2^i} (z_{n-{2^i}})=0
\end{equation}
for $i =0,.., k-1$ on all additive generators $z_{n-{2^i}}$ of $H^{n-{2^i}}(M; \Z/2)$. In our approach for determining whether $M$ is $k$-orientable, we establish the additive generators $z_{n-{2^i}}$ as monomials in the ring generators and expand the Cartan formula on this product to obtain a large expression in terms of squaring operations on ring generators. Finally we substitute the values of these operations into our expression and quotient by all relations between monomials in the top degree to obtain the value of $Sq^{2^i} (z_{n-{2^i}})$.

\section{Orientability of the first Rosenfeld plane $\OCP$}\label{OCP}

In this section we determine that the first Rosenfeld plane $\OCP$ is 2-orientable and not 3-orientable by performing calculations of the action of the Steenrod algebra on its cohomology.
The integral cohomology of this manifold was determined as a ring in \cite{Toda1974}:

\begin{equation}
H^*((\mathbb{O} \otimes \mathbb{C})\mathbb{P}^{2}; \Z)= \mathbb{Z}[t,w]/ (t^9 -3w^2 t, w^3 +15 w^2 t^4 - 9 w t^8),
\end{equation}

where $t$ is in degree 2 and $w$ in degree 8. From this expression we determined an additive basis in every degree using Mathematica. 

\vs

\begin{ttable}\label{additiveOCP}
	
	Below we give an explicit list of the groups $H^*(\OCP; \Z)$ and the values of the monomials in the ring generators $t$ and $w$. 
	The even dimensional groups all vanish.

\begin{tabular}{|l|c|l|l|}
	\hline
	$*$ & $H^* (\OCP; \Z)$   & generators & relations\\
	\hline
	0 & $\Z$  & 1 & \\
	2 & $\Z$  & $t$ & \\
	4 & $\Z$  & $t^2$ &\\
	6 & $\Z$  & $t^3 $ &\\
	8 & $\Z^2$  & $t^4$, $w$ &  \\
	10 & $\Z^2$  & $t^5$, $tw$ &  \\
	12 & $\Z^2$  & $t^6$, $t^2w$ &  \\
	14 & $\Z^2$  & $t^7$, $t^3 w$ &  
	\\
	16 & $\Z^3$ &  $t^8$, $t^4w$, $w^2$ & 
	\\
	18 & $\Z^2$  & $t^5w$, $t w^2$  &  $t^9 = 3 \, t w^2$
	\\
	20 & $\Z^2$  & $t^6w$, $t^2 w^2$  &  $t^{10} = 3 \, t^2 w^2$
	\\
	22 & $\Z^2$  & $t^7w$, $t^3 w^2$  &  $t^{11} = 3 \, t^3 w^2$
	\\
	24 & $\Z^2$  & $\gamma_1$, $\gamma_2$  &  
	$t^{12} = 3 \, \gamma_1 + 9 \, \gamma_2$;  	
	$t^8 w = 2 \, \gamma_1 + 5 \, \gamma_2$;  	\\ & & &
	$t^4 w^2 =  \gamma_1 + 3 \, \gamma_2$; 		
	$w^3 = 3 \, {\gamma_1}$.
	\\
	26 & $\Z$ & $\gamma_3$  &  
	$t^{13} = 78 \, \gamma_3$;  
	$t^9 w = 45 \, \gamma_3$; 	\\ & & &
	$t^5 w^2 = 26 \, \gamma_3$; 
	$t w^3 = 15 \, \gamma_3.$\\	
	28 & $\Z$ & $t\gamma_3$  &  
	$t^{14} = 78 \, t\gamma_3$;  	
	$t^{10} w = 45 \, t\gamma_3$;	\\ & & &
	$t^6 w^2 = 26 \, t\gamma_3$; 		
	$t^2 w^3 = 15 \, t\gamma_3.$\\	
	30 & $\Z$ & $t^2\gamma_3$  &  
	$t^{15} = 78 \, t^2\gamma_3$; 
	$t^{11} w = 45 \, t^2\gamma_3$;	\\ & & &
	$t^7 w^2 = 26 \, t^2\gamma_3$; 	
	$t^3 w^3 = 15 \, t^2\gamma_3.$\\	
	32 & $\Z$ & $t^3\gamma_3$  &  
	$t^{16} = 78 \, t^3\gamma_3$;  	
	$t^{12} w = 45 \, t^3\gamma_3$;	\\ & & &
	$t^8 w^2 = 26 \, t^3\gamma_3$; 	
	$t^4 w^3 = 15 \, t^3\gamma_3.$; \\ & & &
	$w^4 = 9 \, t^3\gamma_3 $\\	
	\hline
\end{tabular}
\end{ttable}
\vs

As the integral cohomology is concentrated in even degrees, the  mod 2 cohomology is given by 
\begin{equation}
H^*((\mathbb{O} \otimes \mathbb{C})\mathbb{P}^{2}; \Z/2)= H^*((\mathbb{O} \otimes \mathbb{C})\mathbb{P}^{2}; \Z) \otimes \Z/2. 
\end{equation}

We write $t'$ and $w'$ for the mod 2 reductions of the generators $t$ and $w$. The calculations are done by considering a general expression for the Steenrod operations, for which we subsequently substitute the coefficients as established in \cite{Ishitoya1992}.

\begin{lem}\label{Sq}
	The action of the Steenrod algebra on $H^*((\mathbb{O} \otimes \mathbb{C})\mathbb{P}^{2}; \Z/2)$ is fully determined by four coefficients $\alpha, \beta, \gamma, \delta \in \mathbb{Z}/2$, where:
	\begin{eqnarray}
	Sq^2 w' &=& \alpha t'^5 + \beta t'w';\\
	Sq^4 w' &=& \gamma t'^6 + \delta t'^2 w'.
	\end{eqnarray}
\end{lem}
\begin{proof}[Proof of Lemma \ref{Sq}]
	
	As there is no odd degree cohomology, all odd squares are zero.
	The action of $Sq^i$ on any class can be decomposed into its action on $t'$ and $w'$ by use of the Cartan formula. Hence the entire action of the Steenrod algebra is determined by $Sq^2 t' = t'^2$, $Sq^2 w' $, $Sq^4 w'$ and $Sq^8 w' = w'^2$. $H^{10}((\mathbb{O} \otimes \mathbb{C})\mathbb{P}^{2}; \Z/2)$ is spanned by $t'^5$ and $t'w'$, therefore $Sq^2 w' $ is some linear combination of these. Similarly $H^{12}((\mathbb{O} \otimes \mathbb{C})\mathbb{P}^{2}; \Z/2)$ is spanned by $t'^6$ and $t'^2 w'$.
\end{proof}

We list the signature of the manifold, which follows from the explicit computations.

\begin{lem}\label{signature}
	The signature of $\OCP$ is 3.
\end{lem}
\begin{proof}[Proof of Lemma \ref{signature}]
	Computer calculations show that $$H^{32}(\OCP; \Z)\cong\Z$$ with $t^{16}=78$, $t^{12}w=45$, $t^{8}w^2=26$, $t^4 w^3=15$ and $w^4=9$. $$H^{16}(\OCP; \Z)\cong\Z^3$$ spanned by $t^8$, $t^4 w$ and $w^2$ hence in this basis the cup square is given by
	\begin{equation}
	\left(
	\begin{matrix}
	78 & 45 & 26 \\
	45 & 26 & 15 \\
	26 & 15 & 9
	\end{matrix}
	\right)
	\end{equation}
	which has three positive eigenvalues.
\end{proof}

The orientability of $\OCP$ can be established by computing the action of the Steenrod squares reaching up to the top dimension.  

\begin{prop}\label{SqOCP2}
	For $\alpha, \beta, \gamma, \delta \in \mathbb{Z}/2$ as in Lemma \ref{Sq}, $(\mathbb{O} \otimes \mathbb{C})\mathbb{P}^{2}$ is 2-orientable if $\beta=1$ and 3-orientable if also $\delta=1$.
\end{prop}
\begin{proof}
	As $H^{1}(\OCP; \Z)=0$, we have $w_1=0$ hence $\OCP$ is 1-orientable.

	Computer calculations show that $$H^{30}(\OCP; \Z) \cong \Z,$$with $t^{15}=78$,  $t^{11}w=45$, $t^{7}w^2=26$ and $t^3 w^3=15$. In order to check whether $w_2=0$, we check whether $$Sq^2:H^{30}(\OCP; \Z/2) \rightarrow H^{32}(\OCP; \Z/2)$$ vanishes. Denote by $t'$ and $w'$ the images of $t$ and $w$ under the ring homomorphism given by tensoring $H^{*}(\OCP; \Z)$ with $\Z/2$. Then the  generator of $H^{30}(\OCP; \Z/2)$ can be represented as $t'^{11} w'$ or as $t'^3 w'^3$. Applying $Sq^2$ to this and using the Cartan formula:
	
	\begin{eqnarray}
	Sq^2  (t'^{11} w') &=& Sq^2 (t'^{11})  w' + t'^{11} Sq^2 (w')\nonumber\\
	&=& t'^{12} w' + t'^{11} \left(\alpha t'^5 + \beta t'w'\right) \nonumber\\
	&=&  (1+ \beta)\, t'^{12} w' + \alpha t'^{16} \nonumber\\
	&=& (1+ \beta)\, t'^{12} w' \nonumber\\
	&=& (1+ \beta)
	\end{eqnarray}
	
	And similarly
	
	\begin{eqnarray}
	Sq^2  (t'^3 w'^3) &=& (1+ \beta) t'^{4} w'^3 \nonumber\\
	&=& (1+ \beta).
	\end{eqnarray}
	
	Both $ t'^{12} w' $ and $ t'^{4} w'^3$ represent the generator of $H^{32}(\OCP; \Z/2)$, hence $w_2=0$ if and only if $\beta =1$.
	
	$$H^{28}(\OCP; \Z)\cong \Z$$ with $t^{14}=78$,  $t^{10}w=45$, $t^{6}w^2=26$ and $t^2 w^3=15$, hence $$H^{28}(\OCP; \Z/2) \cong \Z/2$$ generated by  $t'^{10}w'$ or equivalently $t'^2 w'^3$.
	
	\begin{eqnarray}
	Sq^4  (t'^{10} w') &=& Sq^4  (t'^{10} ) w' + Sq^2 ( t'^{10} ) Sq^2 ( w' ) + t'^{10} Sq^4  ( w')\nonumber\\
	&=& {10 \choose 2} \left(Sq^2 (t')\right)^2 t'^8 w' + 0 + t'^{10} (\gamma t'^6 + \delta t'^2 w') \nonumber\\
	&=& (1+ \delta) \, t'^{12} w' + \gamma t'^{16}\nonumber\\
	&=&(1+ \delta) \, t'^{12} w' \nonumber\\
	&=&(1+ \delta)
	\end{eqnarray}
	
	and similarly
	
	\begin{eqnarray}
	Sq^4  (t'^2 w'^3) &=& (1+ \delta)\, t'^{4} w'^3 \nonumber\\
	&=& (1+ \delta).
	\end{eqnarray}

	Hence $\OCP$ is 3-orientable if and only if $\beta=1$ and $\delta=1$.

\end{proof}

\begin{thm}\label{ThmOCP}
	$(\mathbb{O} \otimes \mathbb{C})\mathbb{P}^{2}$ is 2-orientable (spin) and not 3-orientable.
\end{thm}
\begin{proof}
	By \cite{Ishitoya1992}, $\alpha=\beta=\gamma=1$ and $\delta=0$. Hence by Proposition \ref{SqOCP2}, $\OCP$ is 2-orientable but not 3-orientable.
\end{proof}

\section{Orientability of the second Rosenfeld plane $\OHP$}\label{OHP}

\subsection{Recalling what is known about $\OHP$}

For $(\mathbb{O} \otimes \mathbb{H})\mathbb{P}^{2}$ the mod 2 cohomology was established as a ring by \cite{Nakagawa2001}, and it is generated by five classes in degrees 2, 3, 12, 16 and 20, modulo an ideal generated by twelve relations. In \cite{Nakagawa2001} also the action of $Sq^1$, $Sq^2$, $Sq^4$ and $Sq^8$ was determined on all generators, modulo twelve undetermined coefficients.
We here perform calculations that show that $(\mathbb{O} \otimes \mathbb{H})\mathbb{P}^{2}$ is at least 3-orientable, but because of two of the undetermined coefficients it cannot be established from this information whether $\OHP$ is 4-orientable.

\begin{thm}[\cite{Nakagawa2001}]\label{Nakagawa}
	
	\begin{equation}
	H^* (\OHP, \Z/2) = \Z/2 \left[y_2,y_3, y_{12}, y_{16}, y_{20} \right] /J
	\end{equation}
	
	where $J$ is an ideal generated by twelve homogeneous relations:
	\begin{equation*}
	J=
	\left(
	\begin{array}{l}
	y_3^3, 
	y_{16}y_2 + y_{12}y_3^2 + y_2^6 y_3^2, 
	y_{16}y_3, 
	y_{12}^2 y_2 + y_{12}y_2^4y_3^2 + y_{20} y_3^2, \\
	y_{12}^2y_3, 
	y_{12} y_{16} + y_2^{14} + y_{12} y_2^5 y_3^2 + y_2^{11} y_3^2, 
	y_{12}^3 + y_{16} y_{20} + y_2^5y_{20} y_3^2,\\ 
	y_{12}^2 y_{16} + y_{20}^2 + y_{12} y_2^{11} y_3^2, 
	y_{12}^2 y_{20} + y_{12} y_2^{13} y_3^2 + y_{12} y_2^3y_{20} y_3^2, \\
	y_{12} y_{16}^2 + y_{12} y_2^{13}y_3^2, 
	y_{16}^3 + y_{12} y_{16} y_{20} + y_{12} y_2^5y_{20} y_3^2, 
	y_{16}^2y_{20} + y_2^{13}y_{20} y_3^2
	\end{array}\right)
	\end{equation*}

The following is known about Steenrod operations on the generators.
	
	$Sq^1 y_i=0$ unless $i=2$, in which case 
	\begin{eqnarray}
	Sq^1 y_2 &=& y_3.\nonumber
	\end{eqnarray}

We have that
	\begin{eqnarray*}
		Sq^2 (y_{12}) &=& y_2^7 + y_2 y_{12} + y_2^4 y_3^2,\\
		Sq^4 (y_{12}) &=& y_2^8 + y_2^2 y_{12} + \alpha' y_2^5 y_3^2,\\
		Sq^8 (y_{12}) &=& y_{20} + y_2^4 y_{12} + \alpha'' y_2^7 y_3^2+ \beta'' y_2 y_3^2 y_{12},
	\end{eqnarray*}
	for some coefficients $\alpha'$, $\alpha''$, $\beta'' \in \Z/2$,
	\begin{eqnarray*}
		Sq^2 (y_{16}) &=& 0,\\
		Sq^4 (y_{16}) &=& y_2^7 y_3^2,\\
		Sq^8 (y_{16}) &=& y_{12}^2 + \gamma'' y_2^9 y_3^2+ \delta'' y_2^3 y_3^2 y_{12},
	\end{eqnarray*}
	for some coefficients $\gamma''$, $\delta'' \in \Z/2$, and
	\begin{eqnarray*}
		Sq^2 (y_{20}) &=& y_2^{11} + y_2 y_{20} + \mu y_2^8 y_3^2 + \nu y_2^2 y_3^2 y_{12},\\
		Sq^4 (y_{20}) &=& y_{12}^2 + y_2^6 y_{12} + \mu' y_2^9 y_3^2+ \nu' y_2^3 y_3^2 y_{12},\\
		Sq^8 (y_{20}) &=& y_{12}y_{16} + y_2^8 y_{12} + \lambda'' y_2^{11} y_3^2+  \mu'' y_2^5 y_3^2 y_{12}+ \nu'' y_2 y_3^2 y_{20},
	\end{eqnarray*}
	for some coefficients $\mu$, $\nu$, $\mu'$, $\nu'$, $\lambda''$, $\mu''$, $\nu'' \in \Z/2$.  
\end{thm}

In order to determine the entire action of the Steenrod algebra on the cohomology of $\OHP$, one would need to establish the values of the twelve unknown coefficients above and compute $Sq^{16} y_{20}$.

\subsection{3-orientability of $\OHP$}

The orientability of $\OHP$ can be established by calculating the actions of the Steenrod squares that reach up to the top dimension. An additive basis for the cohomology is listed in section \ref{additiveOHP}, and we will moreover need the values of all monomials in the top degree 64 which is given in section \ref{degree64}. A full list of additive relations can be found online$^\text{1}$.
This description lists all monomials in the five ring generators $y_2$, $y_3$, $y_{12}$, $y_{16}$ and $y_{20}$ up to degree 64, listing which of these are zero in the cohomology and which are equated. 



\begin{lem}
	$\OHP$ is 1-orientable.
\end{lem}
\begin{proof}
	The first generator of the mod 2 cohomology of $\OHP$ is $y_2$ of rank 2, hence $H^1(\OHP, \Z/2) = 0$. Therefore $w_1=0$.
\end{proof}

\begin{thm}
	$\OHP$ is 2-orientable.
\end{thm}
\begin{proof}
	In order to establish 2-orientability of $\OHP$, we will compute the action of $Sq^2$ on $H^{62}(\OHP; \Z/2)$, which is generated by $y_2 y_{20}^3$, as can be read off in the table listed in section \ref{additiveOHP}. The Cartan formula for $Sq^2$ on a 4-term product has 10 terms. Applying this using the Mathematica code and reducing modulo 2 yields:
	\begin{eqnarray}
	Sq^2 (y_2 y_{20}^3) &=& \sum_{i+j+k+l = 2} Sq^i y_2 \,Sq^j y_{20}\, Sq^k y_{20}\, Sq^l y_{20}\nonumber\\
	&=& (Sq^1 y_2) y_{20}^2 (Sq^1 y_{20}) + y_2 y_{20} (Sq^1 y_{20})^2 \\
	& & \hspace{60pt} + (Sq^2 y_2) y_{20}^3 + y_2 y_{20} (Sq^2 y_{20}).\nonumber
\end{eqnarray}
	
We now substitute the values of the squaring operations from Theorem \ref{Nakagawa} to obtain
\begin{eqnarray}
Sq^2 (y_2 y_{20}^3) &=& y_2 y_{20}^3 + y_2 y_{20} (y_2^{11} + y_2 y_{20} + \mu y_2^8 y_3^2 + \nu y_2^2 y_3^2 y_{12}) \nonumber \\
&=& y_2^{12} y_{20}^2 + \mu y_2^9 y_3^2 y_{20}^2 + \nu y_2^3 y_3^2 y_{12} y_{20}^2 \nonumber\\
&=& 0,
\end{eqnarray}	
where the last equality follows from the fact that all three monomials vanish in $H^{64}(\OHP; \Z/2)$ as we can read off in section \ref{degree64}.
	
Hence we see that the $Sq^2$ to the top dimension vanishes, thus $w_2=0$. As $\OHP$ is a smooth manifold, $w_3=0$ as well.
\end{proof}

\begin{thm}
	$\OHP$ is 3-orientable.
\end{thm}
\begin{proof}
We see from the table in section \ref{additiveOHP} that $H^{60}(\OHP; \Z/2) \cong \Z/2$ generated by $y_{20}^3$. The Cartan formula for applying $Sq^4$ on a product of three generators has 15 terms. We apply the Cartan formula in Mathematica and reduce mod 2, and subsequently substitute the squaring operations from Theorem \ref{Nakagawa}.
	\begin{eqnarray}
Sq^4 (y_{20}^3) &=& \sum_{i+j+k = 4} Sq^i y_{20} \,Sq^j y_{20}\, Sq^k y_{20} \nonumber \\
	&=& (Sq^1 y_{20})^2 Sq^2 y_{20} + y_{20} (Sq^2 y_{20})^2 + y_{20}^2 (Sq^4 y_{20}) \nonumber \\
	&=& y_{20}\left(y_2^{11} + y_2 y_{20} + \mu y_2^8 y_3^2 + \nu y_2^2 y_3^2 y_{12} \right)^2 \nonumber \\
&&	\hspace{30pt} + y_{20}^2
	\left( y_{12}^2 + y_2^6 y_{12} + \mu' y_2^9 y_3^2+ \nu' y_2^3 y_3^2 y_{12} \right) \nonumber\\
&=&	
	y_2^{22} y_{20} 
	+ \mu^2 y_2^{16} y_3^4 y_{20} 
	+\nu^2 y_2^4 y_3^4 y_{12}^2 y_{20}
	+ \mu' y_2^9 y_3^2 y_{20}^2 \nonumber\\
&&  + y_2^6  y_{12} y_{20}^2
	+ \nu' y_2^3 y_3^2 y_{12}y_{20}^2
	+ y_{12}^2 y_{20}^2 + y_2^2 y_{20}^3 \nonumber\\
&=& 0,
\end{eqnarray}
where the last equality follows from the fact that all monomials equal zero except for the last two which cancel one other, as we can read off from the table in section \ref{degree64}.

Hence the action of $Sq^4$ up to the top dimension is zero, so $w_4 =0$. Then $w_5 = w_6 = w_7=0$ as well.
\end{proof}

\subsection{Potential 4-orientability}
The action of $Sq^3$, $Sq^5$, $Sq^6$ and $Sq^7$ on all generators of the cohomology of $\OHP$ is established in section \ref{smallsquares} and will be used in this section.

\begin{thm}
	$\OHP$ is 4-orientable if $\beta''+ \nu''=1$. Here coefficient $\beta''$ is the prefactor of $y_2 y_3^2 y_{12}$ in $Sq^8 y_{12}$ and $\nu''$ is the prefactor of $y_2 y_3^2 y_{20}$ in $Sq^8 y_{20}$. 
	
\end{thm}
\begin{proof}
	$H^{56}(\OHP; \Z/2)$ has two generators. Generator $y_2^{12} y_{12} y_{20}$ has a unique monomial representing it. The other generator has five representing monomials, one of which reads $y_2^{9} y_3^2 y_{12} y_{20}$. The Cartan formula for $Sq^8 (y_2^{12} y_{12} y_{20})$ has 203490 terms and for $Sq^8 (y_2^{9} y_3^2 y_{12} y_{20})$ it has 125970 terms. Expanding the Cartan formula and substituting in lower squares in Mathematica yields:
	\begin{eqnarray*}
		Sq^8 (y_2^{12} y_{12} y_{20}) &=& (1+ \beta'' + \nu'') y_2 y_{20}^3\\
		Sq^8 (y_2^{9} y_3^2 y_{12} y_{20}) &=& 0.
	\end{eqnarray*}
	Hence $Sq^8$ vanishes on both generators of $H^{56}(\OHP; \Z/2)$ precisely if $\beta'' + \nu''=1$, in which case $\OHP$ is 4-orientable.
\end{proof}

\subsection{Additive basis for the mod 2 cohomology of $\OHP$}\label{additiveOHP}

\vs

\begin{ttable} 
In the table below, $b_i$ is the dimension of $H^i (\OHP; \Z/2)$.
We give a possible set of generators for the cohomology in each degree.
	
\begin{tabular}{|l|c|l|}
	\hline
	$i$ & \,  $b_i$ \,   & generators \\
	\hline
	0 & 1  & 1 \\
	1 & 0  &   \\
	2 & 1  & $y_2$ \\
	3 & 1  & $y_3 $ \\
	4 & 1  & $y_2^2$   \\	
	5 & 1  & $y_2 \, y_3$   \\
	6 & 2  & $y_2^3$, $y_3^2$    \\
	7 & 1  & $y_2^2\, y_3$   \\	
	8 & 2 &  $y_2^4$, $y_2 \, y_3^2$ \\
	9 & 1  &  $y_2^3\, y_3 $\\
	10 & 2  &  $y_2^5$, $y_2^2\, y_3^2$\\
	11 & 1  & $y_2^4\, y_3$ \\
	12 & 3 & $y_{12}$, $y_2^6$, $y_2^3 \, y_3^2$\\
	13 & 1 & $y_2^5\, y_3 $\\
	14 & 3 & $y_2^4 \, y_3^2$, $y_2\, y_{12}$, $y_2^7$\\
	15 & 2 & $y_2^6\, y_3$, $y_3\, y_{12}$\\
	16 & 4 & $y_2^2\, y_{12}$, $y_2^8$, $y_2^5\, y_3^2$, $y_{16}$\\
	\hline
\end{tabular}

\vs

\begin{tabular}{|l|c|l|}
	\hline
	$*$ & \,  $b_i$ \,    & generators \\
	\hline
	17 & 2  &  $y_2^7\, y_3$, $y_2 \, y_3 \, y_{12}$   \\
	18 & 4 & $y_3^2 \, y_{12}$, $y_2^9$, $y_3^2 \, y_{12} + y_2\, y_{16}$, $y_2^3\, y_{12}$\\
	19 & 2 & $y_2^2\, y_3\, y_{12}$, $y_2^8\, y_3$ \\
	20 & 5 &  $y_2^4\, y_{12}$, $y_2^2\, y_{16}$, $y_2\, y_3^2\, y_{12}$, $y_2^{10}$, $y_{20}$\\
	21 & 2 &  $y_2^3\, y_3\, y_{12}$, $y_2^9\, y_3$\\
	22 & 5 &  $y_2^3\, y_{16}$, $y_2^5\, y_{12}$, $y_2^2\, y_3^2\, y_{12}$, $y_2^{11}$, $y_2\, y_{20}$\\
	23 & 3 &  $y_2^4\, y_3\, y_{12}$, $y_2^{10}\, y_3$, $y_3\, y_{20}$\\
	24 & 6 &  $y_{12}^2$, $y_2^3\, y_3^2\, y_{12}$, $y_2^{12}$, $y_2^4\, y_{16}$, $y_2^6\, y_{12}$, $y_2^2\, y_{20}$\\
	25 & 3 &   $y_2\, y_3\, y_{20}$, $y_2^5\, y_3\, y_{12}$, $y_2^{11}\, y_3$ \\
	26 & 6 &  $y_3^2\, y_{20}$, $y_2\, y_{12}^2 + y_2^5\, y_{16} + y_3^2\, y_{20}$, $y_2\, y_{12}^2$, $y_2^7\, y_{12}$, $y_2^{13}$, $y_2^3\, y_{20}$\\
	27 & 3 & $y_2^2\,y_3\,y_{20}$, $y_2^6\,y_3\,y_{12}$, $y_2^{12}\,y_3$    \\
	28 & 6 &  $y_2^8\,y_{12}$, $y_{12}\,y_{16}$, $y_2^4\,y_{20}$, $y_2^2\,y_{12}^2$, $y_2^2\,y_{12}^2 + y_2^6\,y_{16} + y_2\,y_3^2\,y_{20}$, $y_2^6\,y_{16}$ \\
	29 & 3 &  $y_2^{13}\,y_3$, $y_2^3\,y_3\,y_{20}$, $y_2^7\,y_3\,y_{12}$ \\
	30 & 5 & $y_2^7\,y_{16}$, $y_2\,y_{12}\,y_{16} + y_2^2\,y_3^2\,y_{20}$, $y_2^9\,y_{12}$, $y_2\,y_{12}\,y_{16}$, $y_2^5\,y_{20}$\\
	31 & 2 & $y_2^8\,y_3\,y_{12}$, $y_2^4\,y_3\,y_{20}$    \\
	32 & 7  & $y_2^8\,y_{16}$, $y_{16}^2$, $y_2^8\,y_{16} + y_2^2\,y_{12}\,y_{16}$, $y_{12}\,y_{20}$, $y_2^2\,y_{12}\,y_{16} + y_2^3\,y_3^2\,y_{20}$, \\
 & &	$y_2^{10}\,y_{12}$, $y_2^6\,y_{20}$ \\   
	\hline
\end{tabular}

\vs

\begin{tabular}{|l|c|l|}
	\hline
	$*$ & \,  $b_i$ \,    & generators \\
	\hline
	33 & 2  & $y_2^9\,y_3\,y_{12}$, $y_2^5\,y_3\,y_{20}$  \\
	34 & 5  & $y_2^{11}\,y_{12}$, $y_2\,y_{12}\,y_{20}$, $y_2^3\,y_{12}\,y_{16} + y_2^4\,y_3^2\,y_{20}$, $y_2^3\,y_{12}\,y_{16}$, $y_2^7\,y_{20}$ \\
	35 & 3 & $y_2^{10}\,y_3\,y_{12}$, $y_3\,y_{12}\,y_{20}$, $y_2^6\,y_3\,y_{20}$\\
	36 & 6 &   $y_2^5\,y_3^2\,y_{20}$, $y_2^4\,y_{12}\,y_{16} + y_2^5\,y_3^2\,y_{20}$, $y_2^{12}\,y_{12}$, $y_2^8\,y_{20}$, $y_{16}\,y_{20}$, $y_2^2\,y_{12}\,y_{20}$ \\
	37 & 3 &    $y_2^{11}\,y_3\,y_{12}$, $y_2\,y_3\,y_{12}\,y_{20}$, $y_2^7\,y_3\,y_{20}$ \\
	38 & 6 &  $y_3^2\,y_{12}\,y_{20} + y_2\,y_{16}\,y_{20}$, $y_2^5\,y_{12}\,y_{16}$, $y_2^{13}\,y_{12}$, $y_2^3\,y_{12}\,y_{20}$, $y_2\,y_{16}\,y_{20}$, $y_2^9\,y_{20}$\\
	39 & 3 &    $y_2^2\,y_3\,y_{12}\,y_{20}$, $y_2^8\,y_3\,y_{20}$, $y_2^{12}\,y_3\,y_{12}$ \\ 
	40 & 6 &   $y_2\,y_3^2\,y_{12}\,y_{20}$, $y_2^{10}\,y_{20}$, $y_2^2\,y_{16}\,y_{20}$, $y_{12}^2\,y_{16}$, $y_{12}^2\,y_{16} + y_{20}^2$, $y_2^4\,y_{12}\,y_{20}$\\
	41 & 3 &  $y_2^9\,y_3\,y_{20}$, $y_2^{13}\,y_3\,y_{12}$, $y_2^3\,y_3\,y_{12}\,y_{20}$\\
	42 & 5 & $y_2^2\,y_3^2\,y_{12}\,y_{20} + y_2^3\,y_{16}\,y_{20} + y_2\,y_{20}^2$, $y_2^5\,y_{12}\,y_{20}$, \\
	& & $y_2^2\,y_3^2\,y_{12}\,y_{20} + y_2^3\,y_{16}\,y_{20}$, $       y_2^2\,y_3^2\,y_{12}\,y_{20}$, $y_2^{11}\,y_{20}  $
	\\
	43 & 2 & $y_2^4\,y_3\,y_{12}\,y_{20}$, $y_2^{10}\,y_3\,y_{20} $
	\\
	44 & 5 & $y_2^6\,y_{12}\,y_{20}$, $y_2^4\,y_{16}\,y_{20}$, $y_{12}^2\,y_{20} + y_2^2\,y_{20}^2, 
	y_2^{12}\,y_{20}$, $y_{12}^2\,y_{20} + y_2^4\,y_{16}\,y_{20} $
	\\
	45 & 2 &  $y_2^5\,y_3\,y_{12}\,y_{20}$, $y_2^{11}\,y_3\,y_{20}$ 
	\\
	46 & 4 &   $y_2\,y_{12}^2\,y_{20}$, $y_2\,y_{12}^2\,y_{20} + y_2^5\,y_{16}\,y_{20}$, $y_2^7\,y_{12}\,y_{20}$, $y_2^{13}\,y_{20}$
	\\
	47 & 2 & $y_2^6\,y_3\,y_{12}\,y_{20}$, $y_2^{12}\,y_3\,y_{20}   $
	\\
	48 & 4 & $       y_{12}\,y_{16}\,y_{20}$, $y_2^6\,y_{16}\,y_{20}$, $y_2^2\,y_{12}^2\,y_{20} + y_{12}\,y_{16}\,y_{20}$, $y_2^8\,y_{12}\,y_{20}$
	\\ 
		\hline
\end{tabular}

\vs

\begin{tabular}{|l|c|l|}
	\hline
	$*$ & \,  $b_i$ \,    & generators \\
	\hline
	49 & 2  & $y_2^7\,y_3\,y_{12}\,y_{20}$, $y_2^{13}\,y_3\,y_{20}$ \\
	50 & 3  &   $y_2\,y_{12}\,y_{16}\,y_{20}$, $y_2^7\,y_{16}\,y_{20} + y_2\,y_{12}\,y_{16}\,y_{20}$, $y_2^9\,y_{12}\,y_{20}   $\\	
	51 & 1  &   $ y_2^8\,y_3\,y_{12}\,y_{20}$\\	
	52 & 3  &   $ y_2^2\,y_{12}\,y_{16}\,y_{20} + y_{12}\,y_{20}^2$, $y_2^2\,y_{12}\,y_{16}\,y_{20}$, $y_2^{10}\,y_{12}\,y_{20}  $\\
	53 & 1  &  $y_2^9\,y_3\,y_{12}\,y_{20}$ \\	
	54 & 2  &  $y_2^{11}\,y_{12}\,y_{20}$, $y_2^3\,y_{12}\,y_{16}\,y_{20} $ \\	
	55 & 1  &  $y_2^{10}\,y_3\,y_{12}\,y_{20}$ \\
	56 & 2  &  $y_2^{12}\,y_{12}\,y_{20}$, $y_2^4\,y_{12}\,y_{16}\,y_{20}  $ \\	
	57 & 1  &   $y_2^{11}\,y_3\,y_{12}\,y_{20}$ \\	
	58 & 2  &   $y_2^{13}\,y_{12}\,y_{20}$, $y_2^5\,y_{12}\,y_{16}\,y_{20}  $\\
	59 & 1  &   $y_2^{12}\,y_3\,y_{12}\,y_{20} $\\	
	60 & 1  &   $y_{20}^3 $\\	
	61 & 1  &   $ y_2^{13}\,y_3\,y_{12}\,y_{20}$\\
	62 & 1  &   $y_2\,y_{20}^3$\\
	63 & 0  &   \\
	64 & 1  &   $y_2^2\,y_{20}^3$\\
	\hline
\end{tabular}

\end{ttable}

\subsection{List of non-zero monomials in degree 64}\label{degree64}

\begin{ttable}\label{64monomials}
Below we list the non-zero monomials of degree 64 in the ring generators $y_2$, $y_3$, $y_{12}$, $y_{16}$ and $y_{20}$ of $H^* (\OHP; \Z/2)$. There are 123 64-dimensional monomials in total, out of which 11 equal the generator of $H^{64} (\OHP; \Z/2)$ and the others vanish.

\begin{tabular}{|l c l|}
	\hline
Degree 64 &    &  \\
	\hline
$y_2^{16} y_{12} y_{20} $ &=& $  y_2^2 y_{20}^3           $ \\ $
y_2^{14} y_{12}^3 $ &=& $  y_2^2 y_{20}^3         $ \\ $
y_2^{14} y_{16} y_{20} $ &=& $  y_2^2 y_{20}^3        $ \\ $
y_2^{13} y_3^2 y_{12} y_{20} $ &=& $  y_2^2 y_{20}^3      $ \\ $
y_2^{10} y_{12}^2 y_{20} $ &=& $  y_2^2 y_{20}^3       $ \\ $
y_2^8 y_{12} y_{16} y_{20} $ &=& $  y_2^2 y_{20}^3            $ \\ $
y_2^2 y_{20}^3 $ &=& $  y_2^2 y_{20}^3     $ \\ $     
y_{12}^4 y_{16} $ &=& $  y_2^2 y_{20}^3  $ \\ $         
y_{12}^2 y_{20}^2 $ &=& $  y_2^2 y_{20}^3     $ \\ $    
y_{12} y_{16}^2 y_{20} $ &=& $  y_2^2 y_{20}^3     $ \\ $            
y_{16}^4 $ &=& $  y_2^2 y_{20}^3$\\
\hline
\end{tabular}

\end{ttable}

\subsection{Determining smaller squares}\label{smallsquares}

In this section we calculate the remaining actions of the Steenrod operations up to $Sq^8$ on the generators, namely the action of $Sq^3$, $Sq^5$, $Sq^6$ and $Sq^7$. Mathematica was used in some places for the expansion of the Cartan formula on larger products and to automatically replace the previously established smaller squaring operations.

\subsubsection{Computing the action of $Sq^3$ on the generators}\label{square3}

\begin{eqnarray*}
	Sq^3 y_3 &=& y_3^2\\
	&&\\
	Sq^3 y_{12} &=& Sq^1 Sq^2 y_{12}\\
	&=& Sq^1 (y_2^7 + y_2 y_{12} + y_2^4 y_3^2)\\
	&=& Sq^1 (y_2^7) +  Sq^1 (y_2 y_{12}) +Sq^1 ( y_2^4 y_3^2)\\
	&=& y_2^6 y_3 + y_3 y_{12} + 0 + Sq^1(y_2^4) y_3^2 + y_2^4 Sq^1(y_3^2)\\
	&=& y_2^6 y_3  + y_3 y_{12} \\
	&&\\
	Sq^3 y_{16} &=& Sq^1 Sq^2 y_{16}\\
	&=& 0\\
	&&\\
	Sq^3 y_{20} &=& Sq^1 Sq^2 y_{20}\\
	&=& Sq^1 (y_2^{11}+ y_2 y_{20} + \mu y_2^8 y_3^2 + \nu y_2^2 y_3^2 y_{12})\\
	&=& Sq^1 (y_2) y_2^{10}+ Sq^1(y_2) y_{20} + y_2 Sq^1(y_{20}) + \mu Sq^1(y_2^8 y_3^2) + \nu Sq^1(y_2^2 y_3^2 y_{12})\\
	&=&  y_2^{10} y_3 + y_3 y_{20}  + \mu Sq^1(y_2^8) y_3^2 + \nu Sq^1(y_2^2) y_3^2 y_{12}\\
	&=&  y_2^{10} y_3 + y_3 y_{20}  
\end{eqnarray*}

\subsubsection{Computing the action of $Sq^5$ on the generators}

\begin{eqnarray*}
	Sq^5 y_{12} &=& Sq^1 Sq^4 y_{12}\\
	&=& Sq^1 (y_2^8 +y_2^2 y_{12} + \alpha' y_2^5 y_3^2)\\
	&=&\alpha' y_2^4 y_3^3\\
	&&\\
	Sq^5 y_{16} &=& Sq^1 Sq^4 y_{16}\\
	&=& Sq^1 (y_2^7 y_3^2)\\
	&=& y_2^6 y_3^3\\
	&&\\
	Sq^5 y_{20} &=& Sq^1 Sq^4 y_{20}\\
	&=& Sq^1 (y_{12}^2 + y_2^6 y_{12} + \mu' y_2^9 y_3^2 + \nu' y_2^3y_3^2 y_{12})\\
	&=& Sq^1 (y_{12}^2) + Sq^1 (y_2^6 y_{12}) + \mu' Sq^1 (y_2^9 y_3^2) + \nu' Sq^1 ( y_2^3y_3^2 y_{12})\\
	&=&  Sq^1 (y_2^6) y_{12} + y_2^6 Sq^1 (y_{12})  + \mu' Sq^1 (y_2^9) y_3^2 + \mu' y_2^9 Sq^1 ( y_3^2) \\
	&& + \nu' Sq^1 ( y_2^3) y_3^2 y_{12} + \nu'  y_2^3 Sq^1 (y_3^2) y_{12} + \nu'  y_2^3 y_3^2 Sq^1 ( y_{12}) \\
	&=& 0 + 0 + \mu' Sq^1 (y_2^9) y_3^2 + 0 + \nu' Sq^1 ( y_2^3) y_3^2 y_{12}+0+0\\
	&=& \mu' y_2^8 y_3^3 + \nu' y_2^2 y_3^3 y_{12}
\end{eqnarray*}

\subsubsection{Computing the action of $Sq^6$ on the generators}

\begin{eqnarray*}
	Sq^6 y_{12} &=& Sq^5 Sq^1 y_{12} + Sq^2 Sq^4 y_{12}\\
	&=&0 + Sq^2 (y_2^8 +y_2^2 y_{12} + \alpha' y_2^5 y_3^2)\\
	&=& (Sq^1 (y_2^4))^2 + Sq^2 (y_2^2 y_{12}) + \alpha' Sq^2 (y_2^5 y_3^2)\\
	&=& 0 + y_2^9 + y_2^6 y_3^2 + y_2^3 y_{12} + y_3^2 y_{12} + \alpha' y_2^6 y_3^2  \hspace{40pt}  \text{(Using Mathematica)}\\
	&=& y_2^9 + (1+ \alpha' ) y_2^6 y_3^2 + y_2^3 y_{12} + y_3^2 y_{12} \\
	&&\\
	Sq^6 y_{16} &=& Sq^5 Sq^1 y_{16} + Sq^2 Sq^4 y_{16}\\
	&=&0 + Sq^2 (y_2^7 y_3^2)\\
	&=& Sq^2(y_2^7) y_3^2 + Sq^1(y_2^7) Sq^1(y_3^2) + y_2^7 Sq^2(y_3^2)\\
	&=&  (y_2^8 +y_2^6 y_3) y_3^2 + 0 + 0\\
	&=&  y_2^8 y_3^2 + y_2^6 y_3^3\\
	&&\\
	Sq^6 y_{20} &=& Sq^5 Sq^1 y_{20} + Sq^2 Sq^4 y_{20}\\
	&=&0 + Sq^2 (y_{12}^2 + y_2^6 y_{12} + \mu' y_2^9 y_3^2 + \nu' y_2^3y_3^2 y_{12})\\
	&=& (Sq^1 y_{12})^2 + Sq^2(y_2^6 y_{12}) + \mu' Sq^2(y_2^9 y_3^2) + \nu' Sq^2(y_2^3 y_3^2 y_{12})\\
	&=& 0 + y_2^{13} + y_2^{10}y_3^2 + y_2 ^7 y_{12} + y_2^4 y_3^2 y_{12} + \mu'(y_2^{10}y_3^2) \\
	&& + \nu'(y_2^{10}y_3^2 + y_2^7 y_3^4 + y_2 y_3^4 y_{12})
	\hspace{80pt}  \text{(Using Mathematica)}\\
	&=& y_2^{13} + (1+\mu' + \nu')  y_2^{10}y_3^2 + y_2 ^7 y_{12} + y_2^4 y_3^2 y_{12} + 
	\nu'(y_2^7 y_3^4 + y_2 y_3^4 y_{12})
\end{eqnarray*}

\subsubsection{Computing the action of $Sq^7$ on the generators}\label{square7}

\begin{eqnarray*}
	Sq^7 y_{12} &=& Sq^1 Sq^6 (y_{12})\\
	&=& Sq^1 (y_2^9 + (1+ \alpha' ) y_2^6 y_3^2 + y_2^3 y_{12} + y_3^2 y_{12})\\
	&=& Sq^1 (y_2^9) + (1+ \alpha' ) Sq^1 (y_2^6 y_3^2) + Sq^1 (y_2^3 y_{12}) + Sq^1 (y_3^2 y_{12})\\
	&=& y_2^8 y_3 + 0 + y_2^2 y_3 y_{12} + 0\\
	&=&  y_2^8 y_3  + y_2^2 y_3 y_{12} \\
	&&\\
	Sq^7 y_{16} &=& Sq^1 Sq^6 (y_{16})\\
	&=& Sq^1 (y_2^8 y_3^2 + y_2^6 y_3^3)\\
	&=& Sq^1 (y_2^8 y_3^2) + Sq^1 (y_2^6 y_3^3)\\
	&=& 0\\
	&&\\
	Sq^7 y_{20} &=& Sq^1 Sq^6 (y_{20})\\
	&=& Sq^1 (y_2^{13} + (1+\mu' + \nu')  y_2^{10}y_3^2 + y_2 ^7 y_{12} + y_2^4 y_3^2 y_{12} + 
	\nu' y_2^7 y_3^4 + \nu' y_2 y_3^4 y_{12}) \\
	&=& y_2^{12}y_3 + 0 + y_2 ^6 y_3 y_{12} + 0 +  \nu' y_2^6 y_3^5 + \nu' y_3^5 y_{12} \\
	&=& y_2^{12}y_3 +  y_2 ^6 y_3 y_{12} +  \nu' y_2^6 y_3^5 + \nu' y_3^5 y_{12}
\end{eqnarray*}

\bibliographystyle{alpha}
\bibliography{library}

\begin{thebibliography}{Och81}

\bibitem[Bae02]{Baez2002}
John Baez.
\newblock {The Octonions}.
\newblock {\em Bull. Amer. Math. Soc.}, 39(2):145--205, 2002.

\bibitem[HL11]{Huang2011}
Y.~Huang and N.~C. Leung.
\newblock A uniform description of compact symmetric spaces as {Grassmannians}
  using the magic square.
\newblock {\em Mathematische Annalen}, 350(1):79--106, 2011.

\bibitem[Hoe18]{Hoekzema2017}
Renee~S. Hoekzema.
\newblock {Manifolds with Odd Euler Characteristic and Higher Orientability}.
\newblock {\em International Mathematics Research Notices}, 07 2018.
\newblock rny154.

\bibitem[Ish92]{Ishitoya1992}
K.~Ishitoya.
\newblock {Squaring operations in the Hermitian symmetric spaces}.
\newblock {\em J. Math. Kyoto Univ.}, 32(1):235--244, 1992.

\bibitem[KS17]{Kennard2017}
L.~Kennard and Z.~Su.
\newblock On dimensions supporting a rational projective plane.
\newblock {\em {Journal of Topology and Analysis}}, 2017.
\newblock https://doi.org/10.1142/S1793525319500237.

\bibitem[Nak01]{Nakagawa2001}
M.~Nakagawa.
\newblock The mod 2 cohomology ring of the symmetric space {$E\mathrm{VI}$}.
\newblock {\em Journal of Mathematics of Kyoto University}, 41(3):535--556,
  2001.

\bibitem[Och81]{Ochanine1981}
S.~Ochanine.
\newblock {Signature modulo 16, invariants de Kervaire
  g{\'{e}}n{\'{e}}ralis{\'{e}}s et nombres caract{\'{e}}ristiques dans la
  K-th{\'{e}}orie r{\'{e}}elle}.
\newblock {\em M{\'{e}}moires de la {Soci\'et\'e} {Math\'ematique} de France},
  5:1--142, 1981.

\bibitem[TW74]{Toda1974}
Hirosi Toda and Takashi Watanabe.
\newblock The integral cohomology rings of {$ F_4/T $ and $ E_6/T$}.
\newblock {\em J. Math. Kyoto Univ.}, 14(2):257--286, 1974.

\end{thebibliography}

\end{document}